\documentclass[11pt,a4paper]{article}
\usepackage{authblk}
\usepackage{amsmath}
\usepackage{amssymb}
\usepackage{amsthm}
\usepackage{amsfonts}
\usepackage{enumerate}
\usepackage{pstricks}
\usepackage{pst-plot}
\usepackage{pst-poly}
\usepackage{algorithm}
\usepackage{algorithmic}
\usepackage{mathtools} 
\usepackage{url}
\usepackage{booktabs}
\usepackage{verbatim}
\usepackage{hyperref}
\usepackage{breakurl,url}
\usepackage[margin=1.15in]{geometry}
\newcommand{\vr}[1]{{{#1}}}

\newcommand{\mna}[1]{{\mathcal{#1}}}

\newcommand{\ovr}[1]{\mbox{$\overline{\vr{#1}}$}} 
\newcommand{\uvr}[1]{\mbox{$\underline{\vr{#1}}$}}

\newcommand{\R}[0]{{\mathbb{R}}}

\newcommand{\Z}[0]{{\mathbb{Z}}}

\def\eps{{\varepsilon}}
\newcommand{\mmid}[0]{;\,}		

\newcommand{\seznam}[1]{{\{1, \ldots, {#1}\}}}

\def\clqq{``}
\def\crqq{''}
\def\quo#1{\clqq{}#1\crqq{}}  

\newcommand{\stl}[0]{{\mbox{subject to}\ \ }}
\newcommand{\st}[0]{{\ \ \mbox{subject to}\ \ }}
\DeclareMathOperator{\sgn}{sgn}		

\DeclareMathOperator{\diag}{diag}	
\DeclareMathOperator{\conv}{conv} 
\DeclarePairedDelimiter\parentheses{\lparen}{\rparen}   

\makeatletter
	\newcommand\myparagraph{%
    \@startsection{paragraph}{4}{0mm}%
        {-\baselineskip}%
		{-0.7\baselineskip}%
        {\normalfont\normalsize\bfseries}}
\makeatother

\newtheorem{theorem}{Theorem}

\newtheorem{proposition}{Proposition}

\newtheorem{corollary}{Corollary}

\newtheorem{observation}{Observation}
\theoremstyle{definition}

\newtheorem{example}{Example}

\begin{document}

\title{Absolute value linear programming}

\author[1]{Milan Hlad\'{i}k}

\author[2,3]{David Hartman}

\affil[1]{Department of Applied Mathematics, Faculty of Mathematics and Physics, Charles University, Malostransk\'{e} n\'{a}m. 25, Prague 1, 11800, Czech Republic}

\affil[2]{Computer Science Institute of Charles University, Faculty of Mathematics and Physics, Charles University, Malostransk\'{e} n\'{a}m. 25, Prague 1, 11800, Czech Republic, hartman@iuuk.mff.cuni.cz}

\affil[3]{Institute of Computer Science of the Czech Academy of Sciences, Czech Academy of Sciences, Pod Vod\'{a}renskou v\v{e}\v{z}\'{i} 271/2, Prague 8, 18207, Czech Republic}

\date{}
\maketitle

\begin{abstract}
We deal with linear programming problems involving absolute values in their formulations, so that they are no more expressible as standard linear programs. The presence of absolute values causes the problems to be nonconvex and nonsmooth, so hard to solve. In this paper, we study fundamental properties on the topology and the geometric shape of the solution set, and also conditions for convexity, connectedness, boundedness and integrality of the vertices. Further, we address various complexity issues, showing that many basic questions are NP-hard to solve. We show that the feasible set is a (nonconvex) polyhedral set and, more importantly, every nonconvex polyhedral set can be described by means of absolute value constraints. We also provide a necessary and sufficient condition when a KKT point of a nonconvex quadratic programming reformulation solves the original problem.
\end{abstract}


\section{Introduction}

Mangasarian~\cite{Man2007} introduced absolute value programming as mathematical programming problems involving absolute values. So far, researchers have paid primarily attention to absolute value equations. More general systems or even optimization problems were studied quite rarely; some of the few works include \cite{Man2015b,YamFuk2014}. Our aim is to change this focus and turn the attention to linear programs with absolute values.

\myparagraph{Notation}
Given a matrix $A$, we use $A_{i*}$ for its $i$-th row and $A_{*i}$ for its $i$-th column. Next, $\diag(v)$ is the diagonal matrix with entries given by vector~$v$, $I_n$ is the identity matrix of size $n\times n$, $e_i$ is its $i$-th column and $e=(1,\dots,1)^T$ is the vector of ones (with convenient dimension). Given a set $\mna{M}$, we use $\conv\mna{M}$ for a convex hull of $\mna{M}$.
The sign of a real $r$ is $\sgn(r)=1$ if $r\geq0$ and $\sgn(r)=-1$ otherwise. 
The positive and negative parts of a real $r$ are defined $r^+=\max(r,0)$ and $r^-=\max(-r,0)$, respectively. 
For vector or matrix arguments, the absolute value, the sign function, the positive and negative parts are applied entrywise.

\myparagraph{Absolute value linear programming}
We introduce an absolute value LP problem in the form
\begin{align}\label{avlp}
\max\ c^Tx \st Ax-D|x|\leq b,
\end{align}
where $c\in\R^n$, $b\in\R^m$ and $A,D\in\R^{m\times n}$. By $f^*$ we denote the optimal value, and by $\mna{M}$ we denote the feasible set. Throughout the paper we assume that $D$ is nonnegative:
\begin{quote}
\textbf{Assumption.} $D\geq0$.
\end{quote}

Notice that our assumptions are made without loss of generality for the following reason. The objective function is considered as a linear function since otherwise it can be transformed into the constraints by standard techniques. Equations can also be split to double inequalities by standard means (this need not be the best way from the numerical point of view, but for mathematical analysis we can do it with no harm). 

Nonnegativity of matrix $D$ also does not cause harm to generality. If this is not the case, we write $D=D^+-D^-$, where $D^+\geq0$ and $D^-\geq0$ are the entrywise positive and negative parts of $D$, respectively. Problem \eqref{avlp} then reads 
\begin{align}\label{avlp1}
\max\ c^Tx \st Ax-D^+|x|+D^-|x|\leq b,
\end{align}
which is equivalent to
\begin{align}\label{avlp2}
\max\ c^Tx \st Ax-D^+|x|+D^-y\leq b,\ -y\leq x\leq y.
\end{align}
Due to nonnegativity of $D^-$ is this transformation equivalent: If $x$ solves~\eqref{avlp1}, then $x$ and $y\coloneqq|x|$ solves~\eqref{avlp2}. Conversely, if $x,y$ solves~\eqref{avlp2}, then $x$ solves~\eqref{avlp1}. Problem \eqref{avlp2} follows the structure of \eqref{avlp}, which concludes the explanation.

\myparagraph{Roadmap}
The aim of this paper is to address the fundamental solvability, geometric and computational properties of the problem. In particular, the paper is organized as follows. Section~\ref{ssMotiv} provides a motivation, showing that many problems can be naturally expressed as absolute value LP problems. In Section~\ref{sBasicProp}, we study the geometric structure of the feasible set and give some conditions for boundedness, solvability, connectedness and convexity. We also address the computational complexity issues and propose a certain type of duality. Section~\ref{sNonconvPoly} is devoted to the relation of the feasible set and general (nonconvex) polyhedral sets; we show that every polyhedral set admits an absolute value description. Section~\ref{sIntegr} handles integrality of the vertices of~\eqref{avlp1}. In Section~\ref{sQp}, we consider a quadratic programming reformulation of the absolute value LP problem and provide a characterization when KKT points are optimal solutions of \eqref{avlp1}. Eventually, Section~\ref{sBstab} focuses on the special situation of the so called basis stability, in which the problem is efficiently solvable.

\subsection{Motivation}\label{ssMotiv}

Many (computationally hard) problems can easily be reformulated by means of absolute value LP. In this section, we mention some of them.

\myparagraph{Absolute value equations}
The feasibility problem
\begin{align*}
Ax+|x|=b
\end{align*}
is called the absolute value equation and has attracted the attention of many researchers in recent years \cite{Hla2018b, ManMey2006, Man2007, MooKet2021a, Pro2009, Roh2012e, ZamHla2021a, ZamHla2020aa}. 
Despite its simple formulation, the problem is NP-hard~\cite{Man2007} to solve. 

Obviously, it can be solved by absolute value LP since problem \eqref{avlp} is more general.

Further, as observed many times (see \cite{ManMey2006,Man2007,Pro2009}), the problem of absolute value equations is equivalent to the standard linear complementarity problem. Therefore, \eqref{avlp} has the potential to handle various optimization problems with complementarity constraints.

\myparagraph{Integer linear programming}
Consider a 0-1 integer linear program
\begin{align}\label{intLinP}
\max\ c^Tx \st Ax\leq b,\ x\in\{0,1\}^n.
\end{align}
The problem equivalently states
\begin{align*}
\max\ c^Tx \st Ax\leq b,\ |2x-e|=e,
\end{align*}
or,
\begin{align}\label{intLinPasAVLP}
\max\ c^Tx \st Ax\leq b,\ |y|=e,\ 2x-y=e,
\end{align}
which is an absolute value linear program.

In view of this transformation, many other NP-hard problems are directly reformulated by means of absolute value LP. This is particularly the case for problems arising in graph theory, including the maximum clique, the maximum cut, the vertex cover, the maximum matching, or the graph coloring problem. 

\myparagraph{Disjunctive programming \cite{Bal2018}}
Since $\min(a,b)=\frac{1}{2}(a+b-|a-b|)$, we can formulate the disjunctive inequality
\begin{align*}
f(x)\leq 0 \ \ \vee\ \ g(x)\leq0
\end{align*}
as the absolute value inequality
\begin{align*}
f(x)+g(x)-|f(x)-g(x)|\leq0.
\end{align*}
The formula for the minimum is recursively expanded, e.g.,
\begin{align*}
\min(a,b,c)&=\min(a,\min(b,c))\\
	  &=\frac{1}{4}\big(2a+b+c-|b-c|-\big|2a-b-c+|b-c|\big|\big).
\end{align*}
Thus disjunctions of more terms is easily extended. The expression becomes rather cumbersome, but with help of additional variables a convenient form is derived.

For equations, we can effectively handle disjunctions of more than one equation. To be  concrete, consider a disjunction of two systems of equations
\begin{align*}
f_1(x)=\dots=f_m(x)=0 \ \ \vee\ \ g_1(x)=\dots=g_{\ell}(x)=0.
\end{align*}
Its absolute value reformulation consists of $m\ell$ equations
\begin{align*}
f_i(x)+g_j(x)=|f_i(x)-g_j(x)|,\quad \forall i,j.
\end{align*}

\myparagraph{Interval linear programming \cite{Hla2012a,Roh2006b}}
Let
\begin{align*}
[A\pm D]
=[A-D,A+D]
\coloneqq\{\tilde{A}\mmid |\tilde{A}-A|\leq D\}
\end{align*}
be the interval matrix with the midpoint $A$ and the radius~$D$, or alternatively, the lower bound $A-D$ and the upper bound $A+D$.	

Consider a class of linear programs
\begin{align}\label{ilp}
f(\tilde{A})=\max\ c^Tx \st \tilde{A}x\leq b,
\end{align}
where $\tilde{A}\in[A\pm D]$; this is a type of an interval LP problem. By the theory of interval linear programming on the range of the optimal values~\cite{ChinRam2000,Hla2012a,Mra1998,Roh1980}, the value $f^*$ is the best achievable optimal value of the class of LP problems. That is,
\begin{align}\label{ilpOf}
f^*=\max_{\tilde{A}\in[A\pm D]}\ f(\tilde{A}).
\end{align}
Further, the feasible set of \eqref{avlp} is the union of feasible sets of \eqref{ilp},
\begin{align}\label{aviCup}
\{x\mmid Ax-D|x|\leq b\}
=\bigcup_{\tilde{A}\in[A\pm D]} \{x\mmid \tilde{A}x\leq b\}.
\end{align}

\section{Basic properties}\label{sBasicProp}

In this section, we discuss basic properties of problem \eqref{avlp} and particularly of the feasible set
$$
\mna{M}=\mna{M}(b)=\{x\in\R^n\mmid Ax\leq b+D|x|\}.
$$
We use the notation $\mna{M}(b)$ when the right-hand side vector $b$ is subject to some changes. 

First, observe that problem \eqref{avlp} is nonconvex and nonsmooth optimization problem. Even worse, the feasible set $\mna{M}$ can be disconnected. Consider, for example, the constraint $|x|=e$, which characterizes the number of $2^n$ isolated points $(\pm1,\dots,\pm1)^T$.

The problem becomes tractable, provided that we restrict to any orthant; we get rid of the absolute value then, and \eqref{avlp} turns into an LP problem. More concretely, let $s\in\{\pm1\}^n$ and consider the orthant defined by the sign vector~$s$, that is, the orthant $\diag(s)x\geq0$. Within this orthant, the feasible set $\mna{M}$ reads
\begin{align}\label{avlpOrth}
(A-D\diag(s))x\leq b,\ \diag(s)x\geq0
\end{align}
since we substitute $|x|=\diag(s)x$. As a consequence, we have

\begin{observation}\label{obsDecomp}
The feasible set $\mna{M}$ forms a convex polyhedral set inside each orthant. 
\end{observation}

Hence, the feasible set is the union of at most $2^n$ convex polyhedra. As another consequence, we can solve \eqref{avlp} directly by a reduction to $2^n$ LP problems
\begin{align*}
f^*=\max_{s\in\{\pm1\}^n}\ \max\ c^Tx \st (A-D\diag(s))x\leq b,\ \diag(s)x\geq0.
\end{align*}
Obviously, if the $i$th column of $D$ is zero, then we do not need to distinguish the sign of $s_i$ and the overall complexity decreases. Therefore the complexity is to solve $2^{\ell}$ linear programs, where $\ell$ is the number of nonzero columns of~$D$.

This simplification is not artificial since many absolute value linear programs arising from other fields may have a naturally reduced number of non-zero columns of $D$. Consider the integer linear program~\eqref{intLinP}. Using the reformulation~\eqref{intLinPasAVLP} we get the following absolute value linear program in matrix form,
\begin{align*}
\begin{pmatrix} 
 A & 0_{m,n} \\ 2I_n & -I_n \\ -2I_n & I_n \\
 0_n & 0_n \\  0_n & I_n \\  0_n & -I_n 
\end{pmatrix}
\begin{pmatrix} x \\ y \end{pmatrix}
-\begin{pmatrix} 
 0_{m,n} & 0_{m,n} \\ 0_n & 0_n \\ 0_n & 0_n \\
 0_n & I_n \\  0_n & 0_n \\  0_n & 0_n 
\end{pmatrix}
\begin{vmatrix} x \\ y \end{vmatrix}
\leq 
\begin{pmatrix} 
 b \\ e \\ -e \\ -e \\ e \\ e
\end{pmatrix}.
\end{align*}
Even though the number of constraints and variables increased, the number of nonzero columns of $D$ remains the same. Hence the orthant-by-orthant decomposition complexity remains the same, too.

Eventually, the orthant-by-orthant decomposition partially reveals the structure of the vertices of the convex hull of~$\mna{M}$.

\begin{proposition}\label{propVert}
Let $x^*\in\mna{M}$ and $s\coloneqq\sgn(x^*)$. If $x^*$ is a vertex of $\conv\mna{M}$, then it is a vertex of the convex polyhedral sets 
\begin{align}\label{sysPropVert1}
(A-D\diag(s))x\leq b
\end{align}
and
\begin{align}\label{sysPropVert2}
(A-D\diag(s))x\leq b,\ \ \diag(s)x\geq0.
\end{align}
\end{proposition}

\begin{proof}
First notice that $x^*$ is feasible for both \eqref{sysPropVert1} and~\eqref{sysPropVert2}. Since $x^*$ is a vertex of $\conv\mna{M}$, then it must be a vertex of both these subsets. Indeed, \eqref{sysPropVert1} is a subset of $M\subseteq\conv\mna{M}$ since for every $x$ satisfying \eqref{sysPropVert1} we have $Ax-D|x|\leq(A-D\diag(s))x\leq b$.
\end{proof}

\myparagraph{Boundedness}
The orthant-by-orthant decomposition approach applies to boundedness, too. The feasible set $\mna{M}$ is bounded if and only if \eqref{avlpOrth} is bounded for every $s\in\{\pm1\}^n$. In other words, for every $s\in\{\pm1\}^n$, the system
\begin{align*}
(A-D\diag(s))x\leq 0,\ \diag(s)x\geq0
\end{align*}
has only the trivial solution $x=0$. Equivalently, we state it as follows.

\begin{observation}\label{obsBound}
The feasible set $\mna{M}(b)$ is bounded for each $b\in\R^n$ if and only if the system
\begin{align}\label{sysObsBound}
Ax-D|x|\leq0
\end{align}
has only the trivial solution $x=0$.
\end{observation}

\begin{proposition}
It is a co-NP-complete problem to check if $\mna{M}$ is bounded.
\end{proposition}

\begin{proof}
We use a reduction from the Set-Partitioning problem: 
\begin{center}
Given $a\in\mathbb{Z}^n$, exists $x\in\{\pm1\}^n: a^Tx=0$?
\end{center}
We formulate it as
$$
|x|=e,\ a^Tx=0,
$$
which we further rewrite as
\begin{align}\label{sysPrProfNpBound1}
x\leq e,\ -x\leq e,\ a^Tx=0,\ e\leq |x|.
\end{align}
We claim that its feasibility is equivalent to non-trivial feasibility of
\begin{align}\label{sysPrProfNpBound2}
x\leq ey,\ -x\leq ey,\ a^Tx=0,\ ey\leq |x|,\ y\geq0.
\end{align}
Indeed, if $x^*$ solves the former system, then $x^*$ and $y^*\coloneqq1$ solves the latter. Conversely, suppose that $x^*,y^*$ solves \eqref{sysPrProfNpBound2}. If $y^*>0$, then $\frac{1}{y^*}x^*$ solves \eqref{sysPrProfNpBound1}. Otherwise, if  $y^*=0$, then necessarily $x^*=0$ and thus the solution is trivial. Thus, we proved co-NP-hardness, in view of Observation~\ref{obsBound}. 

The certificate for unboundedness of $\mna{M}$ is any feasible solution and a non-trivial solution of \eqref{sysObsBound}. Therefore, the problem is co-NP-complete.
\end{proof}

\myparagraph{Solvability}
To ensure solvability for each right-hand side vector $b\in\R^n$, it is sufficient and necessary that $Ax-D|x|\leq -e$ is solvable.

\begin{proposition}
The feasible set $\mna{M}(b)$ is nonempty for each $b\in\R^n$ if and only if it is nonempty for $b\coloneqq-e$.
\end{proposition}

\begin{proof}
\quo{If.} 
Let $b\in\R^n$ be arbitrary and let $x^*\in\R^n$ be such that $Ax^*-D|x^*|\leq -e$. If $b\geq0$, then $x^*\in\mna{M}(b)$ and we are done. Otherwise, there exists $k$ such that $b_k < 0$. Define $x^{\alpha}\coloneqq \alpha x^*$, where $\alpha=\max_j\frac{b_j}{(Ax^*-D|x^*|)_j}$.
Note that, $(Ax^*-D|x^*|)_i < 0$ for every~$i$. Since there is at least one $b_i < 0$, namely $b_k$, the resulting $\alpha$ is positive. Therefore
$(Ax^{\alpha}-D|x^{\alpha}|) = \alpha (Ax^*-D|x^*|)\leq b$.

\quo{Only if.} 
Obvious.
\end{proof}

However, it turns out that checking feasibility of the only instance with $b=e$ is a computationally hard problem.

\begin{proposition}
It is NP-complete to check $\mna{M}(-e)\not=\emptyset$, that is, feasibility of the system $Ax-D|x|\leq -e$. 
\end{proposition}

\begin{proof}
By \cite{Hla2012b,Roh2006a}, it is NP-hard to checking solvability of the system
\begin{align*}
|Ax|\leq e,\ e^T|x|>1
\end{align*}
in the set of nonnegative positive definite rational matrices. With respect to solvability, we can equivalently write
\begin{align*}
|Ax|< e,\ e^T|x|>1,
\end{align*}
or, by introducing an auxiliary scalar variable $y$,
\begin{align*}
|Ax|< ey,\ e^T|x|>y,\ y>0
\end{align*}
In view of a suitable scaling, we equivalently have
\begin{align*}
|Ax|-ey\leq-e,\ y-e^T|x|\leq -1 ,\ -y\leq-1.
\end{align*}
Eventually, we rewrite it into the canonical form of~$\mna{M}$,
\begin{align*}
\begin{pmatrix} A & -e \\ -A & -e \\ 0^T & 1 \\ 0^T & -1 \end{pmatrix}
\begin{pmatrix} x \\ y \end{pmatrix}
-\begin{pmatrix} 0 & 0 \\ 0 & 0 \\ e^T & 0 \\ 0^T & 0 \end{pmatrix}
\begin{vmatrix} x \\ y \end{vmatrix}
\leq -e.
\end{align*}

Eventually, the certificate for $\mna{M}(-e)\not=\emptyset$ is a solution of $\mna{M}(-e)$. In view of the orthant decomposition, it is a solution of a system of type \eqref{avlpOrth}, so it has a polynomial size.
\end{proof}

The proof also reveals that the problem remains intractable even when $D$ has at most one nonzero row. On the other hand, in view of Observation~\ref{obsDecomp}, the complexity grows in the number of nonzero columns of~$D$. That is, providing the number of nonzero columns of $D$ is fixed, then the problem is polynomially solvable by the orthant-by-orthant decomposition approach.

\myparagraph{Connectedness}
As we observed, the feasible set need not be connected. So one can be interested in conditions on connectedness. 
Clearly, if $b\geq0$, then it is connected via the origin. A stronger condition follows

\begin{proposition}\label{propConnect}
The feasible set $\mna{M}$ is connected if the system of linear inequalities
\begin{align}\label{ineqPropIliConStr}
(A+D)u-(A-D)v\leq b,\ u,v\geq0
\end{align}
is solvable.
\end{proposition}

\begin{proof}
In view of \eqref{aviCup}, the feasible set $\mna{M}$ can be viewed as the united solution set of an interval system of linear inequalities. The rest follows directly from \cite[Prop.~2]{Hla2014d}, which gives a sufficient condition for connectedness in the context of interval inequalities.
\end{proof}



\myparagraph{Convexity}
There are two trivial examples, where the feasible set $\mna{M}$ is convex -- the matrix $D$ is zero, or the whole feasible set lies in one orthant. Nevertheless, the set can sometimes be convex even when it intersects the interiors of at least two orthants and $D\not=0$. These situations are hard to characterize, but in essence they somehow combine the above two trivial examples.

\begin{proposition}\label{propConvNec}
Let $\mna{M}$ be convex and denote by $\mna{M}_s$ the set described by $(A-D\diag(s))x\leq b$. Let $x^1$ and $x^2$, respectively, be any vertices of $\mna{M}_{s^1}$ and $\mna{M}_{s^2}$, corresponding to bases $B^1$ and $B^2$, and such that they lie in the orthants determined by sign vectors $s^1$ and~$s^2$. If $i\in B^1\cap B^2$ and $x^1_jx^2_j<0$ for some~$j$, then $D_{ij}=0$.
\end{proposition}

\begin{proof}
Suppose to the contrary that $D_{ij}>0$ for certain $i,j$. From the assumptions of the proposition,
\begin{align*}
(Ax^1-D|x^1|)_i=\big(Ax^1-D\diag(s^1)x^1\big)_i=b_i,\\
(Ax^2-D|x^2|)_i=\big(Ax^2-D\diag(s^2)x^2\big)_i=b_i.
\end{align*}
Thus, for any strict convex combination $x^*\coloneqq \lambda_1 x^1+\lambda_2 x^2$, where $\lambda_1,\lambda_2>0$ and $\lambda_1+\lambda_2=1$, we have
\begin{align*}
A_{i*}x^*-D_{i*}(\lambda_1|x^1|+\lambda_2|x^2|)=b_i.
\end{align*}
Since $x^1_jx^2_j<0$ and $D_{ij}>0$, we get 
\begin{equation*}
D_{ij}(\lambda_1|x^1_j|+\lambda_2|x^2_j|)
>D_{ij}|\lambda_1x^1_j+\lambda_2x^2_j|.
\end{equation*}
Therefore 
$
A_{i*}x^*-D_{i*}|x^*|>b_i
$
and $x^*$ does not belong to~$\mna{M}$; a contradiction.
\end{proof}

The condition presented in Proposition~\ref{propConvNec} is necessary for convexity of $\mna{M}$, but not sufficient. Consider, for example, the system
\begin{align*}
-2\leq x_1\leq 2,\ 
x_2\geq0,\ 
x_2\leq |x_1|,\ 
x_2\leq 0.5x_1+|x_1|.
\end{align*}
Then the condition is satisfied, see also Figure~\ref{figConvexityofSystems}, but the set characterized by this system is not convex.

\begin{figure}[t]
\centering
\includegraphics[width=0.9\textwidth]{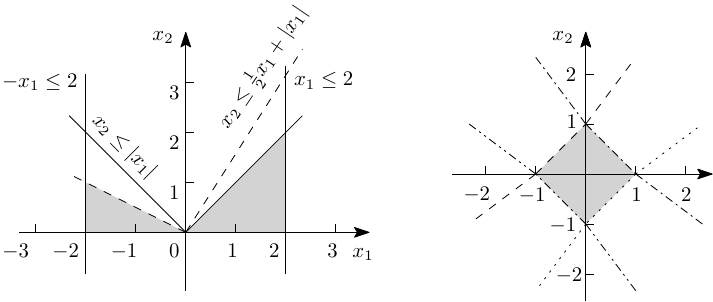}
\caption{Nonconvex $\mna{M}$ with conditions of Proposition~\ref{propConvNec} satisfied (left). Convex $\mna{M}$ intersecting more orthants (right). \label{figConvexityofSystems}}
\end{figure}

The situation where the feasible set $\mna{M}$ is convex and intersects several orthants may happen, for example, when different constraints are active in different orthants. Consider for concreteness the system
\begin{align*}
 10x_1+10x_2-|x|_1-|x|_2&\leq 9,\quad
-10x_1+10x_2-|x|_1-|x|_2\leq 9,\\
 10x_1-10x_2-|x|_1-|x|_2&\leq 9,\quad
-10x_1-10x_2-|x|_1-|x|_2\leq 9.
\end{align*}
It describes the unite ball in the Manhattan norm, see also Figure~\ref{figConvexityofSystems}, so  $\mna{M}$ is convex and intersects all orthants. The following statement formalizes this observation.

\begin{proposition}
Let $\mna{M}$ be convex and let the $i$-th inequality in $\mna{M}$ be active at points $x^1$ and $x^2$. Then for every $j\in\seznam{n}$, the condition $x^1_jx^2_j<0$ implies $D_{ij}=0$.
\end{proposition}

\begin{proof}
From the assumptions of the proposition,
\begin{align*}
A_{i*}x^1-D_{i*}|x^1|=b_i,\quad
A_{i*}x^2-D_{i*}|x^2|=b_i.
\end{align*}
From convexity of $\mna{M}$, we have for any convex combination $\lambda_1 x^1+\lambda_2 x^2$, where $\lambda_1,\lambda_2\geq0$ and $\lambda_1+\lambda_2=1$,
\begin{align*}
A_{i*}(\lambda_1x^1+\lambda_2x^2)-D_{i*}|\lambda_1x^1+\lambda_2x^2|\leq b_i.
\end{align*}
Thus, we derive
\begin{align*}
\lambda_1 D_{i*}|x^1|+\lambda_2 D_{i*}|x^2|
\leq D_{i*}|\lambda_1x^1+\lambda_2x^2|.
\end{align*}
From the triangle inequality, the above holds as equation. Hence for each $j\in\seznam{n}$ we have
\begin{align*}
\lambda_1 D_{ij}|x^1_j|+\lambda_2 D_{ij}|x^2_j|
= D_{ij}|\lambda_1x^1_j+\lambda_2x^2_j|.
\end{align*}
This can happen only if $D_{ij}=0$ or $x^1_jx^2_j\geq0$.
\end{proof}

In \cite{Hla2023u} it was shown intractability of checking convexity of $\mna{M}$, for the particular case of absolute value equations.

\myparagraph{Complexity}
Since the absolute value equation problem $Ax+|x|=b$ is NP-hard \cite{Man2007}, it is also intractable to solve absolute value LP. In particular, it is NP-hard to check for feasibility of~\eqref{avlp}. 
Further, we show that it is also hard to verify that a certain value is the optimal value.

\begin{proposition}
It is an NP-complete problem to check if $f^*=0$.
\end{proposition}

\begin{proof}
We again utilize a reduction from the Set-Partitioning problem: 
\begin{center}
Given $a\in\mathbb{Z}^n$, exists $x\in\{\pm1\}^n: a^Tx=0$?
\end{center}
We formulate it as
$$
|x|=e,\ a^Tx=0.
$$
Consider the absolute value LP problem
\begin{align*}
\max\ a^Tx \st |x|=e,\ a^Tx\leq0.
\end{align*}
Its optimal value is zero if and only if the Set-Partitioning problem is feasible. Therefore it is NP-hard to check if $f^*=0$. The certificate for $f^*=0$ is any solution of the Set-Partitioning problem, which proves NP-completeness.
\end{proof}

\myparagraph{Duality}
The interval linear programming viewpoint \eqref{ilpOf} allows us to introduce a certain kind of duality in absolute value LP (cf.\ duality in interval LP \cite{NovHla2020a,Ser2005,Roh1980}). Let 
\begin{align*}
g(\tilde{A})=\min\ b^Ty \st \tilde{A}^Ty=c,\ y\geq0
\end{align*}
be the dual problem to \eqref{ilp}. Based on weak duality in LP and \eqref{ilpOf} we get weak duality for \eqref{avlp}
\begin{align*}
f^* \leq \max\ g(\tilde{A}) \st \tilde{A}\in[A\pm D].
\end{align*}
Strong duality can be derived under a certain assumption. Basically, we need to ensure strong duality in the LP instances, that is, $f(\tilde{A})=g(\tilde{A})$ for each $\tilde{A}\in[A\pm D]$. Checking this property is known to be co-NP-hard \cite{NovHla2020a}, but there are cheap sufficient conditions. One of them is feasibility of \eqref{ilp} for each $\tilde{A}\in[A\pm D]$, which was addressed in the proof of Proposition~\ref{propConnect}.

\begin{proposition}
If \eqref{ineqPropIliConStr} is feasible, then
\begin{align*}
f^* = \max\ g(\tilde{A}) \st \tilde{A}\in[A\pm D].
\end{align*}
\end{proposition}

Under the same assumption, we have a finite but exponential reduction formula (cf.\ \cite{ChinRam2000,Hla2012a,Mra1998,NovHla2020a,Roh1984})
\begin{align*}
f^* 
&= \max_{s\in\{\pm1\}^n}\ \min\ b^Ty \st (A-D\diag(s))^Ty=c,\ y\geq0,\\
&= \max_{s\in\{\pm1\}^n}\ \max\ c^Tx \st (A-D\diag(s))x\leq b.
\end{align*}

\section{Further geometric properties}\label{sNonconvPoly}

By Observation~\ref{obsDecomp} we know that the system
\begin{align}\label{avls}
Ax-D|x|\leq b
\end{align}
describes a polyhedral set that is convex within each orthant. A natural question is whether it holds the other way round as well: Can every polyhedral set that is convex inside each orthant be described by a system \eqref{avls} without an increase of dimension (additional variables)? The answer is negative.

\begin{example}\label{exCupOrtImposs}
Consider the set
\begin{align*}
M = \{x\in\R^2\mmid x_1\leq-1,\ x_2\geq 1\}
 \cup \{x\in\R^2\mmid -x_1+x_2\leq0,\ x_2\geq 1\},
\end{align*}
which is depicted in Figure~\ref{figExCupOrtImposs}. If this set could be formulated as \eqref{avls}, then there should be an absolute value inequality
\begin{align*}
\alpha|x_1|+\beta|x_2|+\gamma x_1+\delta x_2 \leq 0
\end{align*}
such that it reduces to inequality $-x_1+x_2\leq0$ in the nonnegative orthant. In the nonnegative orthant, the absolute value inequality reads $(\alpha+\gamma)x_1+(\beta+\delta)x_2\leq0$, whence
\begin{align*}
\alpha+\gamma=-1,\ \ \beta+\delta=1.
\end{align*}
Thus the absolute value inequality takes the form
\begin{align*}
\alpha|x_1|+\beta|x_2|+(-1-\alpha)x_1+(1-\beta) x_2 \leq 0.
\end{align*}
In the orthant associated with the sign vector $(-1,1)$, this system reads
\begin{align*}
x_2\leq (1+2\alpha)x_1.
\end{align*}
There is no $\alpha$ such that this inequality is satisfied for every point $x\in M$; note that the problematic points are those on the half-line parallel to~$x_2$. Thus, there is no simple way to express this system using~\eqref{avls}.
\begin{figure}[t]
\centering
\includegraphics[width=0.5\textwidth]{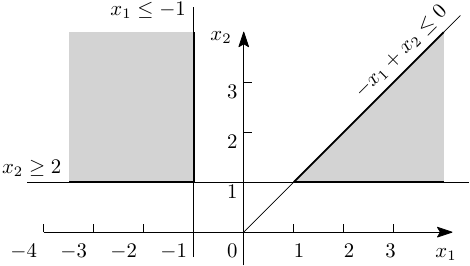}
\caption{(Example~\ref{exCupOrtImposs}) Visualization of the disconnected polyhedral set~$M$. 
\label{figExCupOrtImposs}}
\end{figure}

Nevertheless, we can still reformulate $M$ by means of absolute value inequalities, but on account of increasing the number of variables. With help of an additional variable $x_3$, we describe $M$ by the system 
\begin{align*}
 x_1+x_3-|x_3|\leq-1,\ -x_1+x_2-x_3-|x_3|\leq0,\ x_2\geq 1.
\end{align*}
The case $x_3\geq0$ characterizes the left part of $M$ (as depicted in  Figure~\ref{figExCupOrtImposs}), and the case $x_3\leq0$ characterizes the right part of~$M$.
\end{example}

This example indicates that the problematic polyhedral sets are those that are unbounded and there is an unbounded direction perpendicular to an axis. Indeed, avoiding such cases, we can prove the property to be true.

\begin{theorem}\label{thmPolOrthAsAvls}
Let $M\subseteq\R^n$ be a polyhedral set and convex in each orthant. Suppose there is no  unbounded direction in the boundary of $M$ that is orthogonal to an axis. Then $M$ can be described by means of~\eqref{avls}.
\end{theorem}

\begin{proof}
The set $M$ is characterized by a union of convex polyhedra described by linear inequalities. Consider any inequality $a^Tx\leq b$ from the description of~$M$. Without loss of generality assume that this inequality characterizes a convex part of $M$ in the nonnegative orthant. Consider now the absolute value inequality
\begin{align*}
 (\alpha e+a)^Tx -\alpha e^T|x| \leq b,
\end{align*}
where $\alpha>0$ is large enough. Let us focus on an arbitrary but fixed orthant associated with a sign vector $e\not=s\in\{\pm1\}^n$; the corresponding orthant is characterized by the inequality $\diag(s)x\geq0$. Within this orthant, the absolute value inequality takes the form of
\begin{align*}
 2\alpha \sum_{i\in I}x_i \leq b-a^Tx,
\end{align*}
where $I=\{i\mmid s_i=-1\}$. We claim that any feasible point $x\in M$, $\diag(s)x\geq0$, satisfies this inequality. It is sufficient to prove it for vertices and extremal directions only. If $\sum_{i\in I}x_i=0$, then point $x$ lies on the border of the nonnegative orthant and the inequality $0 \leq b-a^Tx$ obviously holds. Thus we can assume that $\sum_{i\in I}x_i<0$. Since 
\begin{align*}
 -2\alpha \sum_{i\in I}|x_i| = 2\alpha \sum_{i\in I}x_i \leq b-a^Tx
\end{align*}
and $\alpha>0$ is large enough, the inequality is satisfied for any vertex. Hence it remains to inspect the extremal directions. In order that the inequality is violated, there must be an unbounded edge $x^*+\lambda y^*$, $\lambda\geq0$, such that $a^Ty^*>0$ and $y_i^*=0$ for every $i\in I$. However, this means that the edge is orthogonal to the axes $x_i$, $i\in I$; a contradiction.
\end{proof}

The technique described in the proof provides an efficient representation of $M$ by absolute value inequalities. If $M$ is characterized by $m$ facets, then the resulting system \eqref{avls} consists of $m$ inequalities. 
Figure~\ref{figConvexify} illustrates the application of Theorem~\ref{thmPolOrthAsAvls}.

\begin{figure}[t]
\centering
\includegraphics[width=0.5\textwidth]{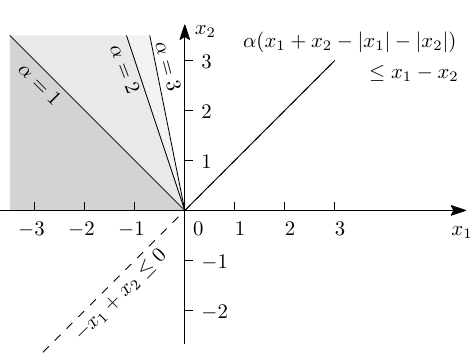}
\caption{Representation of inequality $-x_1+x_2 \leq 0$ from Example~\ref{exCupOrtImposs} in the orthant defined by $s=(-1,1)$ using the techniques from Theorem~\ref{thmPolOrthAsAvls}.\label{figConvexify}}
\end{figure}

\begin{corollary}
Let $M\subseteq\R^n$ be a polyhedral set and convex in each orthant. Suppose that in the orthant associated with sign $s\in\{\pm1\}^n$ it is characterized by $m_s$ linear inequalities. Under the assumption of Theorem~\ref{thmPolOrthAsAvls}, the set $M$ can be described by the system \eqref{avls} with $\sum_{s\in\{\pm1\}^n}m_s$ inequalities.
\end{corollary}

Example~\ref{exCupOrtImposs} indicated that we can overcome the orthogonality assumption of Theorem~\ref{thmPolOrthAsAvls} but on account of additional variables. Indeed, only $n$ additional variables are needed to rewrite any polyhedral set that is convex within any orthant in the form of~\eqref{avls}. In fact, we present a stronger result on reformulation of an arbitrary polyhedral set.

\begin{theorem}\label{thmPolGenAsAvls}
Let $M\subseteq\R^n$ be a polyhedral set described
\begin{align}\label{setThmPolGenAsAvls}
M=\bigcup_{i=1}^m\{x\in\R^n\mmid A^ix\leq b^i\}.
\end{align}
It can be characterized as the absolute value system~\eqref{avls} with at most $\log(m)$ additional variables.
\end{theorem}

\begin{proof}
Suppose first that $m=2^k$ for some natural~$k$. Each convex polyhedral set $\{x\mmid A^ix\leq b^i\}$ can be uniquely associated with a vector $s^i\in\{\pm1\}^k$ by a suitable bijection; notice that $s^i$ is not interpreted as the sign vector of some feasible solutions now. We claim that $M$ is characterized by the system
\begin{align}\label{sysPfThmPolGenAsAvls}
A^ix + \sum_{j=1}^k (s^i_jz_j-|z_j|)e \leq b^i,\quad i=1,\dots,2^k.
\end{align}
If $x\in M$, then the point $x$ satisfies a particular system $A^kx\leq b^k$ for some particular~$k$. We simply put $z\coloneqq \alpha s^k$, where $\alpha>0$ is sufficiently large. Then $(x,z)$ satisfies system~\eqref{sysPfThmPolGenAsAvls}; we need $\alpha>0$ large enough in order that  $(x,z)$ satisfies~\eqref{sysPfThmPolGenAsAvls} for $i\not=k$.

Conversely, let  $(x,z)$ satisfy~\eqref{sysPfThmPolGenAsAvls} and let $s^i\coloneqq\sgn(z)$ be the sign vector of~$z$. Then $x$ lies in the the convex polyhedral set associated with~$s^i$ since 
\begin{align*}
A^ix = A^ix + \sum_{j=1}^k (s^i_jz_j-|z_j|)e \leq b^i.
\end{align*}

If $m$ is not a power of $2$, we proceed analogously; we just omit some of the terms in the above summation in~\eqref{sysPfThmPolGenAsAvls}. For example, if $m=3$, then the corresponding system reads
\begin{align*}
 A^1x + z_1-|z_1| &\leq b^1,\\
 A^2x - z_1-|z_1| + z_2-|z_2| &\leq b^2,\\
 A^3x - z_1-|z_1| - z_2-|z_2| &\leq b^3.
\end{align*}
\end{proof}


\begin{example}
Consider the union of the convex polyhedral sets described as in \eqref{setThmPolGenAsAvls}. 
The set can be reformulated by means of integer linear programming using the constraints and additional continuous and binary variables
\begin{align*}
A^ix^i\leq b^i,\ \ 
x=\sum_{i=1}^m z_ix^i,\ \ 
e^Tz=1,\ \  
z\in\{0,1\}^m.
\end{align*}
The direct way to express it as an absolute value linear system produces
\begin{align*}
A^ix^i\leq b^i,\ \ 
x=\sum_{i=1}^m z_ix^i,\ \ 
e^Tz=1,\ \  
|2z-e|=e.
\end{align*}
This system employs $m(n+1)$ variables, and the reformulation to the canonical form \eqref{avls} increases the number by~$m$ more. In contrast, the technique from Theorem~\ref{thmPolGenAsAvls} requires merely $\log(m)$ new variables.
\end{example}

\section{Integrality}\label{sIntegr}

This section aims to make a link with integer programming. In particular, we are interested in integrality of the vertices of the feasible set~$\mna{M}$. In the theory of integer linear programming, integrality of vertices is related to unimodular and totally unimodular  matrices~\cite{Schr1998}. 
Recall that a matrix $A\in\Z^{m\times n}$ is unimodular if each basis (nonsingular submatrix of order~$m$) has the determinant $+1$ or $-1$~\cite[Sect.~21.4]{Schr1998}. 
Such matrices naturally appear in the context of absolute value programs, too. 

Throughout this section, we assume that $A,D\in\Z^{m\times n}$. By a vertex of $\mna{M}$, we mean a vertex of \eqref{sysPropVert2} for some $s\in\{\pm1\}^n$, but the following characterization is also valid when we employ \eqref{sysPropVert1} instead.

\begin{proposition}\label{propUnimodExp}
The vertices of $\mna{M}$ are integral for every $b\in\Z^m$ if and only if matrix $(A-D\diag(s))^T$ is unimodular for each $s\in\{\pm1\}^n$.
\end{proposition}

\begin{proof}
\quo{If.} 
Let $b\in\Z^m$ and let $x^*$ be a vertex of~$\mna{M}$. By Proposition~\ref{propVert}, $x^*$ is a vertex of the convex polyhedron described by 
$
(A-D\diag(s^*))x\leq b,
$ 
$\diag(s^*)x\geq0,$ 
where $s^*\coloneqq\sgn(x^*)$. Due to unimodularity of $(A-D\diag(s^*))^T$, vertex $x^*$ is integral.

\quo{Only if.} 
We know that matrix $(A-D)^T$ is integral, and without loss of generality we assume that $(A-D)^T$ is not unimodular. Thus, there is a basis $B$ such that $(A-D)_B^{-1}$ is not integral. Suppose some non-integral elements are in the $i$th column of $(A-D)_B^{-1}$ and define $x^*\coloneqq (A-D)_B^{-1} e_i + \alpha e$, where $0<\alpha\in\Z$ is sufficiently large. Then $x^*$ is nonintegral and nonnegative (i.e., it lies in the correct orthant). Define $b\in\Z^m$ such that $b_B\coloneqq e_i+\alpha (A-D)_B e$ and $b_N\coloneqq \beta e$, where $0<\beta\in\Z$ is sufficiently large. Now, $(A-D)_B x^*=(A-D)_B((A-D)_B^{-1} e_i + \alpha e)=b_B$ and $(A-D)_N x^* \leq b_N$. Therefore, $x^*$ is a vertex of $(A-D)x\leq b$, and because it lies in the correct orthant, it is also a vertex of~$\mna{M}$. As shown above, it is non-integral.
\end{proof}

Due to row linearity of the determinant we have that if the condition of Proposition~\ref{propUnimodExp} is satisfied, then matrix $(A-D\diag(s))^T$ is unimodular for each $s\in\{\pm1,0\}^n$. In particular, $A^T$ is unimodular.

\begin{corollary}
If matrix $(A-D\diag(s))^T$ is unimodular for each $s\in\{\pm1\}^n$, then it is unimodular for each $s\in\{\pm1,0\}^n$.
\end{corollary}

\begin{proof}
Let $s^{i+}$, $s^{i-}$ and $s^{i0}$ denote the vector $s$, where we additionally modify $s(i) = +1$, $s(i) = -1$ and $s(i) = 0$, respectively. 
Let $A^{i+} \coloneqq (A - D\diag(s^{i+}))^T$, $A^{i-} \coloneqq (A - D\diag(s^{i-}))^T$ and $A^{i0} \coloneqq (A - D\diag(s^{i0}))^T$. Due to linearity of determinant,  we have $2\det(A^{i0}) = \det(A^{i+}) + \det(A^{i-})$. For $\det(A^{i-}) = \det(A^{i+}) \in \{\pm 1, 0\}$ we get $\det(A^{i0}) \in \{\pm 1,0\}$. For $\det(A^{i-}) = -\det(A^{i+})$ we get $\det(A^{i0}) = 0$. In the case when $|\det(A^{i-}) -\det(A^{i+})| = 1$ we have $\det(A^{i0}) \in \{\pm\frac{1}{2}\}$. However, this is not possible since the matrix $\det(A^{i0})$ is integral and thus its determinant should be integral as well. Using this approach iteratively, we obtain that the matrix $(A-D\diag(s))^T$ is unimodular for each $s\in \{\pm 1, 0\}$. 
\end{proof}

It is known~\cite{Schr1998} that unimodularity of a matrix is a polynomially decidable problem. In our case of the absolute value problem, the characterization of Proposition~\ref{propUnimodExp} is exponential. It is an open problem what is the real complexity of the problem -- is it polynomial or NP-hard? Anyway, for the selection of the sign vector $s$, we cannot reduce the set $\{\pm1\}^n$ to a fixed subset of polynomial size.

\begin{proposition}
There is no subset $\mna{S}\subseteq\{\pm1\}^n$ of size at most $2^{n-1}-1$ such that the condition of Proposition~\ref{propUnimodExp} can be reduced to $s\in\mna{S}$.
\end{proposition}

\begin{proof}
Suppose to the contrary that such a set $\mna{S}$ exists for some~$n$. We first show that there are two vectors $s_1,s_2\in\{\pm1\}^n\setminus\mna{S}$ such that $|s_1-s_2|=2e_i$ for certain $i\in\seznam{n}$. To see it, observe that the set $\{\pm1\}^n$ consists of all vertices of the $n$-dimensional hypercube. It is known that the size of a maximum independent vertex set of the hypercube is $2^{n-1}$ (and the maximum independent vertex set consists of either those vectors that have odd number of minus ones, or those with even number). Since the cardinality of the set $\{\pm1\}^n\setminus\mna{S}$ is at least $2^{n-1}+1$, there must be inside two vertices connected by an edge, and these two vertices $s_1$ and $s_2$ have the required form.

Without loss of generality suppose that $s_1=e$, $s_2=e-2e_1$ and let
\begin{align*}
A=-I_n,\quad 
D=I_n-e_1e_1^T.
\end{align*}
For any $s\in\mna{S}$, the matrix $(A-D\diag(s))^T$ is unimodular; in fact, it is unimodular for each $s\not\in\{s_1,s_2\}$. However, it is not unimodular for $s\in\{s_1,s_2\}$, so the reduction to $\mna{S}$ is not sufficient.
\end{proof}

Regarding the complexity, one polynomially decidable subclass is that where $D$ has fixed rank. We first discuss the case with $D$ having rank one, and then we extend it to an arbitrary fixed rank. In the following, $\|s\|_0$ denotes the $0$-norm of $s$, that is, the number of nonzero entries in~$s$.

\begin{proposition}\label{propUnimodRankOne}
Let $D$ have rank one. Then the characterization of Proposition~\ref{propUnimodExp} is satisfied if and only if matrix $(A-D\diag(s))^T$ is unimodular for each $s\in\{\pm1,0\}^n$ such that $\|s\|_0\leq 2$.
\end{proposition}

\begin{proof}
\quo{If.} 
Let $B\subseteq\seznam{m}$ of size $n$ and $\tilde{A}\coloneqq A^T_B$. Since $D$ has rank one, we can write $D^T_B=-uv^T$ for some vectors $u,v\in\Z^n$.

First, we suppose $B$ is a basis of~$A$, that is, $\tilde{A}$ is nonsingular. Thus,
\begin{align*}
\det\parentheses[\big]{\tilde{A}+uv^T\diag(s)}
=\det(\tilde{A})\big(1+v^T\diag(s)\tilde{A}^{-1}u\big).
\end{align*}
This value is $\pm1$ or $0$ for every $s\in\{\pm1,0\}^n$ such that $\|s\|_0\leq 2$. Considering the case $s = 0$ and the fact that $\tilde{A}$ is nonsingular, we get that $\det(\tilde{A}) = \pm1$. Thus the function $f(s)\coloneqq v^T\diag(s)\tilde{A}^{-1}u$ satisfies $f(s)\in\{0,-1,-2\}$ particularly for every $s\in\{0,\pm e_1,\dots,\pm e_n\}$. Notice that $f(s)$ is linear, so we have $f(0)=0$ and $f(s)>0$ implies $f(-s)<0$. Therefore, function $f(s)$ is constantly zero. In particular, it is zero for each $s\in\{\pm1\}^n$ and  thus $\det\parentheses[\big]{\tilde{A}+uv^T\diag(s)} = \pm1$ for each $s\in\{\pm1\}^n$. Now, apply Proposition~\ref{propUnimodExp}. 

Second, suppose that $\tilde{A}$ is singular, but $\tilde{A}+uv^T\diag(e_i)$ is nonsingular for some $i\in\seznam{n}$. We proceed in the same way as in the previous case; we just substitute $\tilde{A}\equiv\tilde{A}+uv^T\diag(e_i)$.

Eventually, suppose that $\tilde{A}$ and matrices $\tilde{A}+uv^T\diag(e_i)$, $i=1,\dots,n$, are singular. If all matrices in the form $\tilde{A}+uv^T\diag(s)$, $s\in\{\pm1\}^n$, are singular, then we are done. So suppose there is $s^*\in\{\pm1\}^n$ such that $C\coloneqq\tilde{A}+uv^T\diag(s^*)$ is nonsingular; from the assumption, we know that $\det(C) = \pm1$. Now, the function
\begin{align*}
\det\parentheses[\big]{\tilde{A}+uv^T\diag(s)}
&=\det\parentheses[\big]{C+uv^T\diag(s-s^*)}\\
&=\det(C)\big(1+v^T\diag(s-s^*)C^{-1}u\big)
\end{align*}
is zero at $s=0$ and $s=\pm e_i$, $i=1,\dots,n$. Hence the linear function $v^T\diag(s-s^*)C^{-1}u$ is constantly $-1$ at $s=0$ and $s=\pm e_i$, $i=1,\dots,n$. This means that it is constant for each $s\in\R^n$, which contradicts the case $s=s^*$.

\quo{Only if.} 
This is clear from Proposition~\ref{propUnimodExp} and the discussion below it.
\end{proof}

In the statement of Proposition~\ref{propUnimodRankOne}, considering only those vectors $s\in\{\pm1\}^n$ such that $\|s\|_0\leq 1$ would not be sufficient. As an example, let 
$$
A=\begin{pmatrix}0&0\\1&1\end{pmatrix},\quad
D=\begin{pmatrix}1&1\\0&0\end{pmatrix}.
$$
Then the characterization of Proposition~\ref{propUnimodExp} is not satisfied (take, e.g., $s=(1,-1)^T$), but for each $s\in\{\pm1,0\}^n$ such that $\|s\|_0\leq 1$ the matrix $(A-D\diag(s))^T$ is unimodular.

\section{Quadratic programming reformulation}\label{sQp}

A common technique to relax the absolute value $|x|$ is to substitute $x\coloneqq x^1-x^2$, $x^1,x^2\geq0$, and replace $|x|$ with $x^1+x^2$. Herein, $x^1$ approximates the positive part and $x^2$ the negative part of~$x$. In this way, problem \eqref{avlp} is simplified to the linear program
\begin{align}\label{avlpRelaxGen}
\max\ c^Tx^1-c^Tx^2 \st (A-D)x^1-(A+D)x^2\leq b,\ x^1,x^2\geq0,
\end{align}
which provides an upper bound on~$f^*$. The bound can be very weak: If $x^1$ and $x^2$ are feasible solutions of \eqref{avlpRelaxGen}, then $x^1+\alpha e$ and $x^2+\alpha e$ are feasible for every $\alpha\geq0$. Provided each row of $D$ is nonzero, problem \eqref{avlpRelaxGen} is feasible (just take $\alpha$ large enough), even when \eqref{avlp1} is infeasible.

In order to obtain an equivalent reformulation, we include an additional term in the objective function, yielding a quadratic program
\begin{subequations}\label{avlpQp}
\begin{align}
&\max\ c^Tx^1-c^Tx^2-\alpha (x^1)^Tx^2 \\ 
&\,\stl (A-D)x^1-(A+D)x^2\leq b,\ x^1,x^2\geq 0,
\end{align}
\end{subequations}
where $\alpha>0$ is a large constant; by means of Schrijver~\cite{Schr1998} it can be a~priori determined having a polynomial size. The additional term $\alpha(x^1)^Tx^2$ ensures complementarity $(x^1)^Tx^2=0$ so that $x^1-x^2$ is the optimal solution of~\eqref{avlp}. In this case, we say that a solution $(x^1,x^2)$ of \eqref{avlpQp} \emph{yields} an optimum of~\eqref{avlp}.

General nonconvex quadratic programs are hard to solve. In the following theorem, we characterize the class of problem for which any KKT solution automatically produces a feasible solution. In the formulation, we make use of the vector relation $a\gneqq b$ defined as $a\geq b$, $a\not=b$. 

\begin{theorem}\label{thmQpKkt}
In problem \eqref{avlpQp}, for any $b\in\R^n$ each KKT point yields a feasible solution of \eqref{avlp} if and only if 
\begin{align}\label{sysThmKktOptQp}
 |c-A^Tw|\lneqq D^Tw,\ w\geq0
\end{align}
is infeasible.
\end{theorem}

\begin{proof}
First notice that the KKT conditions of the quadratic program  \eqref{avlpQp} read
\begin{subequations}\label{kktThmKktOptQp}
\begin{align}
\label{kktThmKktOptQp1}
-c+\alpha x^2+(A-D)^Tw=u&\geq0,\\
\label{kktThmKktOptQp2}
 c+\alpha x^1-(A+D)^Tw=v&\geq0,\\
 w&\geq0,\\
u^Tx^1=v^Tx^2=w^T\big(b-(A-D)x^1+(A+D)x^2\big)&=0.\label{kktThmKktOptQp4}
\end{align}
\end{subequations}
where \eqref{kktThmKktOptQp4} is the complementary slackness.

\quo{Only if.}
Let $w$ be a solution of system \eqref{sysThmKktOptQp} and define
\begin{align*}
x^1 &\coloneqq \frac{1}{\alpha} \big(-c+(A+D)^Tw\big)^+ \geq0,\quad 
v \coloneqq c+\alpha x^1-(A+D)^Tw \geq0,\\
x^2 &\coloneqq \frac{1}{\alpha} \big(c-(A-D)^Tw\big)^+ \geq0,\quad 
u \coloneqq -c+\alpha x^2+(A-D)^Tw \geq0,\\
b &\coloneqq (A-D)x^1-(A+D)x^2.
\end{align*}
Then the KKT conditions \eqref{kktThmKktOptQp} are satisfied, including the complementarity conditions. In addition, due to the definition of $b$, the pair $(x^1,x^2)$ is feasible to~\eqref{avlpQp}. From the assumption, there is $i$ such that 
$|c_i-(A^Tw)_i| < (D^Tw)_i$. In other words,
\begin{align*}
0<-c_i+(A+D)^T_{i*}w,\quad
0<c_i-(A-D)^T_{i*}w.
\end{align*}
By the definition of $x^1$ and $x^2$ we have $x^1_i>0$ and $x^2_i>0$. Therefore the complementarity $(x^1)^Tx^2=0$ is not satisfied. This also means that the point $x^*\coloneqq x^1-x^2$ does not belong to~$\mna{M}$. To see it, recall that
$$
(A-D)x^1-(A+D)x^2=A(x^1-x^2)-D(x^1+x^2)=b.
$$
In view of $|c_i-(A^Tw)_i| < (D^T)_{i*}w$ we have $D_{*i}\not=0$, which implies $D(x^1+x^2)\gneqq D|x^*|$. Thus $Ax^*-D|x^*|\not\leq b$.

\quo{If.}
Let $x^1,x^2,u,v,w$ satisfy the KKT conditions \eqref{kktThmKktOptQp}. From $u^Tx^1=0$ we have $x^1_i=0$ or $u_i=0$ for each~$i$. The former implies $c_i-((A+D)^Tw)_i\geq0$ and the latter implies $c_i-((A-D)^Tw)_i\geq0$. In any case, we deduce $(A^Tw)_i-c_i\leq (D^Tw)_i$. From $v^Tx^2=0$ we analogously obtain $c_i-(A^Tw)_i\leq (D^Tw)_i$. Thus, in total, we have
\begin{align}\label{sysPfThmKktOptQp}
|c-A^Tw|\leq D^Tw.
\end{align}
Suppose to the contrary that the KKT point $(x^1,x^2)$ does not produce a feasible solution. Thus we have $(x^1)^Tx^2>0$, that is, there is $i$ such that $x^1_i>0$ and $x^2_i>0$. Then $u_i=v_i=0$ and 
\begin{align*}
c_i-((A-D)^Tw)_i=\alpha x^2_i>0,\quad
-c_i+((A+D)^Tw)_i=\alpha x^1_i>0.
\end{align*}
Hence $|c_i-(A^Tw)_i| < (D^Tw)_i$ and \eqref{sysThmKktOptQp} is feasible.
\end{proof}

Notice that solvability of \eqref{sysThmKktOptQp} can be checked in polynomial time by means of linear programming. The system can be stated as
\begin{align*}
-(A+D)^Tw&\leq-c,\ \ 
(A-D)^Tw\leq c,\ \ 
w\geq0,\\
-e^T(A+D)^Tw&\leq-e^Tc-\eps,\ \ 
e^T(A-D)^Tw\leq e^Tc-\eps,
\end{align*}
where $\eps>0$ is small enough with polynomial size (cf.\ \cite{Schr1998}).

Due to NP-hardness of the absolute value LP problem and strong conditions of Theorem~\ref{thmQpKkt}, the system \eqref{sysThmKktOptQp} is often feasible. However, the class of infeasible instances is nontrivial. It comprises not only the case $D=0$, but also the instances for which the optimum $x^*$ satisfies $Ax^*\leq b$, that is, $x^*$ is an optimum of the LP problem $\max \{c^Tx\mmid Ax\leq b\}$.

\begin{proposition}
Suppose that \eqref{avlp} has an optimum. If \eqref{sysThmKktOptQp} is infeasible and $Ax\leq b$ is feasible, then there is an optimum $x^*$ of \eqref{avlp} such that $Ax^*\leq b$. 
\end{proposition}

\begin{proof}
Suppose to the contrary that no optimum $x^*$ of \eqref{avlp} satisfies $Ax^*\leq b$. That is, the system
$$
Ax\leq b,\ \ c^Tx\geq c^Tx^*
$$
is infeasible. By the Farkas lemma, the dual system
\begin{align*}
A^Tw=cz,\ \ (w,z)\geq0,\ \ b^Tw<(c^Tx^*)z
\end{align*}
has a solution $(w^*,z^*)$. If $z^*=0$, then again by the Farkas lemma applied to the resulting system $A^Tw=0$, $w\geq0$, $b^Tw<0$ we obtain that the system $Ax\leq b$ is infeasible; a contradiction. Thus $z^*>0$ and we can assume without loss of generality that $z^*=1$. Hence we have
\begin{align*}
A^Tw^*=c,\ \ w^*\geq0,\ \ b^Tw^*<c^Tx^*.
\end{align*}
Premultiplying inequality $Ax^*-D|x^*|\leq b$ by $w^*$, we get
$$
c^Tx^*-(w^*)^TD|x^*|\leq b^Tw^*.
$$
If $D^Tw^*=0$, then $c^Tx^*\leq b^Tw^*$; a contradiction. Therefore, in view of $D \geq 0$, we have
\begin{align*}
|A^Tw^*-c|=0\lneqq D^Tw^*,
\end{align*}
meaning that \eqref{sysThmKktOptQp} is feasible; a contradiction.
\end{proof}

\section{Special situation of basis stability}\label{sBstab}

We already observed a connection between absolute value LP and interval LP. Utilizing this relation, we can identify a class of problems, which are efficiently solvable -- the so-called basis stable problems. 

In interval LP, basis stability refers to a situation, in which there is a common optimal basis of \eqref{ilp} for each $\tilde{A}\in[A\pm D]$. In this case, the absolute value LP problem is easily resolved. However, there are two drawbacks -- first, the situation is rare, and, second, it is co-NP-hard to check for basis stability~\cite{Hla2014a}. The good news is there there are sufficient conditions that work well \cite{Kon2001,Kra1975}; we adapt them to our problem.

\myparagraph{How to check for basis stability}
Let $B$ be a basis. By $A_B$ we mean a restriction of $A$ to the rows indexed by $B$, and similarly for $A_N$, where $N\coloneqq\seznam{m}\setminus B$ are nonbasic indices. 
Basis $B$ is optimal for the LP problem \eqref{ilp} with a certain $\tilde{A}\in[A\pm D]$ if and only if $\tilde{A}_B$ is nonsingular and the following two conditions hold,
\begin{align}
\label{bStab1}
(\tilde{A}^{-1}_B)^T c &\geq 0,\\
\label{bStab2}
\tilde{A}_N\tilde{A}^{-1}_B b_B &\leq b_N.
\end{align}
To verify basis stability w.r.t.\ basis $B$, we have to check validity of these conditions for every $\tilde{A}\in[A\pm D]$.

A condition for \eqref{bStab1} works as follows. Consider the interval system of linear equations
\begin{align*}
[A\pm D]^T_B y = c.
\end{align*}
There exist many methods to solve it; see \cite{May2017,MooKea2009,Neu1990,Roh2006a}. A solution to such a system is an interval vector $[\uvr{y},\ovr{y}]$ such that $(\tilde{A}^{-1}_B)^T c\in[\uvr{y},\ovr{y}]$ for every $\tilde{A}\in[A\pm D]$. Thus we solve the interval system, and then we just check for $\uvr{y}\geq0$, which shows stability of~\eqref{bStab1}.

For condition \eqref{bStab2}, we proceed similarly. First, solve the interval system of linear equations 
$[A\pm D]_B x = b_B,$ 
and let $[\uvr{x},\ovr{x}]$ be the resulting interval solution. Now, evaluate 
$[\uvr{z},\ovr{z}]\coloneqq \tilde{A}_N[\uvr{x},\ovr{x}]$ by interval arithmetic. Eventually, if $\ovr{z}\leq b_N$, then stability of condition \eqref{bStab2} is verified.

\myparagraph{How to find the optimal value and optimal solution}
Once stability of an optimal basis $B$ is verified, the optimal value $f^*$ can be expressed
as
\begin{align*}
f^* = \max\ c^T\tilde{A}^{-1}_B b_B \st \tilde{A}\in[A\pm D].
\end{align*}
To solve this optimization problem efficiently by means of linear programming, we substitute $y\coloneqq (\tilde{A}^{-1}_B)^T c\geq0$ and write
\begin{align*}
f^* 
 = \max\ b_B^Ty \st \tilde{A}^T_B y = c,\ y\geq0,\ \tilde{A}\in[A\pm D].
\end{align*}
By the properties of the united solution set of interval systems of linear equations~\cite{May2017,Neu1990,Roh2006a}, we can express the feasible set of the above optimization problem as
$$
(\tilde{A}-D)^T_B y \leq c \leq  (\tilde{A}+D)^T_B y,\ y\geq0.
$$
In this way, we obtain an LP formulation for~$f^*$.

\begin{corollary}
Under basis stability with basis $B$, 
\begin{align}
f^* = \max\ b_B^Ty \st (\tilde{A}-D)^T_B y \leq c \leq  (\tilde{A}+D)^T_B y,\ y\geq0.
\label{lpBstabOpt}
\end{align}
\end{corollary}

After computing $f^*$, we determine the optimal solution $x^*$, too. Let $y^*$ be an optimum to~\eqref{lpBstabOpt}. Determine $\tilde{A}_B\in[A\pm D]_B$ such that $\tilde{A}_B^Ty^*=c$, which is an easy task~\cite{Roh2006a}. Finally, we have $x^*=\tilde{A}_B^{-1}b_B$.

\section{Conclusion}

In this paper, we thoroughly investigated geometric and computational-complexity properties of the absolute value LP problems. In particular, we presented  various conditions for convexity, connectedness, boundedness and feasibility of the feasible set. We also investigated the formulation power of absolute value inequalities in characterizing nonconvex polyhedral sets. In linear programming, integrality of vertices relates to unimodular matrices, and in case of absolute value LP problems the unimodularity property extends to matrices of certain form. Absolute value LP problems can be reformulated by means of integer programming or quadratic programming; for the latter, we proposed a necessary and sufficient condition when a KKT point automatically produces feasible solutions of the original problem.

Below, we sum up some of the problems that remain open; they mostly regard the feasible set~$\mna{M}$:
\begin{itemize}
\item
\emph{Necessary and sufficient condition for connectedness of~$\mna{M}$.}

So far, only a simple sufficient condition is known. Some more results would be very desirable because handling disconnectedness in optimization is a hard task.
\item
\emph{Necessary and sufficient condition for convexity of~$\mna{M}$.}

We proposed two necessary conditions, but a complete characterization of convexity is unknown.
\item
\emph{The computational complexity} (polynomial vs.\ NP-hard) of checking integrality of the vertices of~$\mna{M}$ for every~$b\in\Z^m$. The characterization proposed in this paper is exponential in~$n$, which however does not exclude the possibility of a polynomial characterization.
\end{itemize}

\subsubsection*{Acknowledgments.} 
The authors were supported by the Czech Science Foundation Grant P403-22-11117S.


\bibliographystyle{abbrv}
\bibliography{abs_val_prog}

\end{document}